\documentclass[11pt, a4paper, oneside, reqno]{amsart}

\renewcommand\thesection{\arabic{section}}
\renewcommand\thesubsection{\thesection.\arabic{subsection}}

\makeatletter
\def\@settitle{\begin{center}%
		\baselineskip14\p@\relax
		\normalfont\LARGE\scshape\bfseries
		\@title
	\end{center}%
}
\makeatother

\makeatletter

\def\subsection{\@startsection{subsection}{2}%
	\z@{.5\linespacing\@plus.7\linespacing}{.5\linespacing}%
	{\normalfont\large\bfseries}}

\def\subsubsection{\@startsection{subsubsection}{3}%
	\z@{.5\linespacing\@plus.7\linespacing}{.5\linespacing}%
	{\normalfont\itshape}}

\usepackage[colorlinks	= true,
			raiselinks	= true,
			linkcolor	= MidnightBlue,
			citecolor	= Mahogany,
			urlcolor	= ForestGreen,
			pdfauthor	= {Peyman Mohajerin Esfahani},
			pdftitle	= {},
			pdfkeywords	= {},   
			pdfsubject	= {},
			plainpages	= false]{hyperref}
\usepackage[usenames, dvipsnames]{color}
\usepackage{dsfont,amssymb,amsmath,subfig, graphicx,fancyhdr,mdframed}
\usepackage{amsfonts,dsfont,mathtools, mathrsfs,amsthm,wrapfig} 
\usepackage[]{siunitx}
\usepackage{ragged2e}
\let\labelindent\relax
\usepackage{enumitem}

\allowdisplaybreaks
\usepackage{algorithm}
\usepackage[noend]{algpseudocode}

\allowdisplaybreaks
\date{\today}

\patchcmd\@setauthors
  {\MakeUppercase{\authors}}
  {\authors}
  {}{}

\newtheorem{Thm}{Theorem}[section]
\newtheorem{Prop}[Thm]{Proposition}
\newtheorem{Fact}[Thm]{Fact}
\newtheorem{Lem}[Thm]{Lemma}
\newtheorem{Cor}[Thm]{Corollary}

\theoremstyle{remark}
\newtheorem{Rem}[Thm]{Remark}

\theoremstyle{remark}

\theoremstyle{definition}

\newcommand{\R}{\mathbb{R}}
\newcommand{\Z}{\mathbb{Z}}

\newcommand{\ra}{\rightarrow}

\newcommand{\ind}{\mathds{1}}
\newcommand{\Let}{\coloneqq}

\newcommand{\tr}{^\intercal}

\newcommand{\C}{\mathbb{C}}

\newcommand{\EZ}[1][]{e}
\newcommand{\fEZ}[1][]{e}

\newcommand{\shiftop}{\mathfrak{q}}

\graphicspath{{images/}}

\begin{document}
\title{Real-time Fault Estimation for a Class of Discrete-Time Linear Parameter-Varying Systems}


\author{Chris van der Ploeg$^{1,2}$, Emilia Silvas$^{1,2}$, Nathan van de Wouw$^{1,3}$, Peyman Mohajerin Esfahani$^{4}$}
\thanks{1. Department of Mechanical Engineering, Eindhoven University of Technology, The Netherlands}
\thanks{2. The Netherlands Organisation for Applied Scientific Research, Integrated Vehicle Safety Group}
\thanks{3. Department of Civil, Environmental and Geo-Engineering, University of Minnesota, U.S.A}
\thanks{4. Delft Center for Systems and Control, Delft University of Technology, The Netherlands}
\thanks{Peyman Mohajerin Esfahani acknowledges the
support of the ERC grant TRUST-949796.}
\maketitle

\begin{abstract} 
Estimating and detecting faults is crucial in ensuring safe and efficient automated systems. In the presence of disturbances, noise or  varying system dynamics, such estimation is even more challenging. To address this challenge, this article proposes a novel filter to estimate multiple fault signals for a class of discrete-time linear parameter-varying (LPV) systems. The design of such a filter is formulated as an optimization problem and is solved recursively, while the system dynamics may vary over time. Conditions for existence and detectability of the fault are introduced and the problem is formulated and solved using the quadratic programming framework.
{\color{black} We further propose an approximate scheme that can be arbitrarily precise while it enjoys an analytical solution, which supports real-time implementation. 
The method is illustrated and validated on an automated vehicle's lateral dynamics, which is a practically relevant example for LPV systems. The results show that the estimation filter can decouple unknown disturbances and known or measurable parameter variations in the dynamics while estimating the unknown fault.}

\end{abstract}
\section{Introduction}{\color{black}
The problem of fault diagnosis has been an extensively studied topic over the past decades. The detection and estimation of a fault can support an action of the system mitigating the effect of the fault, improving the safety of the system and potential users. In literature, various categories of fault diagnosis methods are elaborated upon,  see~\cite{Ding20081, Gao20153757} and the references therein. 
In the scope of fault {\em detection}, i.e., detecting the presence of a fault and {\em estimation}, i.e., determining the exact magnitude and shape of a fault, choosing between fault-sensitivity, attenuation and decoupling of disturbances and uncertainties is often the most challenging trade-off~\cite{varga2017solving}. The task of {\em isolation} can be seen as a special case of detection and estimation, where all faults can be decoupled from one another using disturbance attenuation techniques, although the complexity of this problem highly depends on the condition of fault isolability~\cite{van2020multiple}.

The class of linear parameter-varying (LPV) systems is often considered in the scope of fault detection and estimation and is particularly suitable for treating non-linear systems with parameter variations as linear systems with time-varying and potentially measurable parameters.  
A class of solutions was defined in literature through the use of linear matrix inequalities (LMI) to robustly formulate the sensitivity problem in an optimization framework using Lyapunov functions~\cite{6082383,Wei2011114,CHEN20178600,Henry2012190,DO20202099,Hamdi2021247,Gomez-Penate2019339}. Therein,  parameter-independent Lyapunov functions~\cite{6082383} are used in a polytopic framework, which due to their time-independent nature could result in conservative solutions~\cite{Wei2011114}. Other works consider the use of parameter-dependent Lyapunov functions for filter synthesis in either a polytopic framework~\cite{CHEN20178600} or in a linear fractional transformation framework~\cite{Henry2012190}. These methods can handle potentially uncertain LPV systems. However, their computational burden is often high and may not guarantee the decoupling of disturbances, unless assumptions are made on the frequency content through the use of a complementary disturbance observer~\cite{DO20202099} (i.e., a proportional integral (PI) observer). In~\cite{Hamdi2021247,Gomez-Penate2019339} a sliding-mode observer is proposed for continuous-time systems, for which the observer gains are synthesised through LMIs. Although these methods are well suited for parameter-varying systems, they could suffer from chattering or singularities, requiring a relaxation of the proposed solution.

A different solution to the LPV fault estimation problem is the use of a geometric approach. By exploiting the known model, disturbance directions that are not of interest can be projected in parameter-varying unobservable subspaces of the fault estimator~\cite{BOKOR2004511}. {\color{black}A nullspace approach, an application of the geometric approach, is proposed in~\cite{2006Nyberg}, which has been extended by a robust formulation for non-linear systems~\cite{Esfahani2016} and parameter-varying systems~\cite{VARGA20116697}. These approaches consider a continuous-time model setting, whereas in this work, amongst other contributions to be mentioned hereafter, we focus on a discrete-time model setting and a closed-form parameter-varying solution.}

\textbf{Our contributions:} In summary, there exist many approaches to the problem of fault detection and estimation for linear parameter-varying systems. Yet, there does not yet exist a solution which could guarantee the decoupling of disturbances, while isolating and estimating the fault of interest in real-time (i.e., having a {\em practically implementable} solution in the form of a discrete filter with {\em low computational burden}). As such, we define our contributions as follows:

 \begin{enumerate}[label=(\roman*)]
    \item \textbf{Parameter-varying filter synthesis:} We propose a novel parameter-varying polynomial decomposition for LPV dynamical systems (Lemma~\ref{lem:poly}), which paves the way for a convex reformulation of the isolation/estimation filter at each time instance (Proposition~\ref{theorem:LPVtheorem}).
    
    \item \textbf{Isolability conditions:} 
    We offer the {\em existence} conditions of an isolation filter via a novel polynomial time-varying matrix construction (Lemma~\ref{lem:poly}).
    This allows for a {\em tractable} evaluation of isolability for the LPV systems.
    
    \item \textbf{Analytical solution:} We further propose an arbitrarily accurate approximation for the original program of the filter design whose solution is analytically available (Corollary~\ref{cor:dualproblem}). This analytical solution allows for {\em implementable} real-time filter synthesis while using valuable practical considerations in the context of LPV systems. 
 \end{enumerate}
The LPV estimation filter is demonstrated on the lateral dynamics of an automated vehicle, a popular illustrative example for LPV fault detection/estimation techniques~\cite{Zhou2020,8269396}. Herein, the estimation challenge is to detect an offset in the steering system, while the vehicle can have a time-varying yet measurable longitudinal velocity.} 

The outline of this work is as follows. First, the problem formulation is provided in Section~\ref{sec:problemformulation}. In Section~\ref{sec:theorem}, the design of the LPV estimation filter is provided. Moreover, the problem is considered from a practical perspective, showing that the synthesis of such a estimation filter can be implemented by the use of generic computational tools, e.g., matrix inversion. In Section~\ref{sec:simulation}, the estimation filter is demonstrated by application to an example of the lateral dynamics of an automated vehicle. Finally, Section~\ref{sec:conclusion} draws conclusions and proposes future work. 


 
 
\section{Model description and preliminaries}\label{sec:problemformulation}
In this section, a class of LPV systems is introduced along with some basic definitions. The model is an LPV extension of the differential-algebraic equations (DAE) class of discrete-time models introduced in~\cite{2006Nyberg} and is described as
\begin{align}\label{eq:DAE}
   H(w_k,\shiftop)[x]+L(w_k,\shiftop)[z]+F(w_k,\shiftop)[f]=0,
\end{align}
where $\shiftop$ represents the shift operator (i.e., $\shiftop[x(k)]=x(k+1)$), $x,z,f,w$ represent discrete-time signals indexed by the discrete time counter $k$, taking values in~$\mathbb{R}^{n_{x}},\mathbb{R}^{n_{z}},\mathbb{R}^{n_{f}},\mathbb{R}^{n_{w}}$. The matrices $H(w_k,\shiftop), L(w_k,\shiftop), F(w_k,\shiftop)$ are parameter-varying polynomial functions in the variable $\shiftop$, depending on the parameter signal $w$ with $n_r$ rows and $n_x,\:n_z,\:n_f$ columns, respectively. {\color{black} Finally, $w$ represents a scheduling parameter of which the explicit relationship with time is unknown a priori, but the parameter is measurable in real-time and takes values from a compact set~$\mathcal{W}\subseteq\mathbb{R}^{n_{w}},\forall k$.} 
{\color{black}The signal $z$ is assumed to be known or measurable up to the current time $k$ and consist of, e.g., the known or measurable inputs and outputs to and from the system. The signals $x$ and $f$ are unknown and represent the state of the system and the fault, where $f$ is not restricted to any particular location (e.g., sensor or actuator fault).}
\begin{Rem}[Non-measurable scheduling parameters or model uncertainty]
    Several suggestions exist in literature, in the scope of geometric nullspace-based estimation filters, which can be used in making these filters suitable for {\em non-measurable} scheduling parameters $w_k$~\cite[Section 3.3]{VARGA20116697}. Note, that the proposed approximation methods for unknown parameters $w_k$ are directly applicable in the results from this work. 
\end{Rem}The model~\eqref{eq:DAE} encompasses a large class of parameter-varying dynamical systems, an example of which is a set of LPV state-space difference equations. This example will be used in the simulation study and can be derived from~\eqref{eq:DAE} by starting from the following LPV difference equations:
{\color{black}\begin{align}\label{eq:ldif}
	\begin{cases}
		G(w_k)X(k+1) = A(w_k)X(k)+B_u(w_k)u(k) + B_d(w_k)d(k)+B_f(w_k)f(k),\\
		y(k) = C(w_k)X(k)+D_u(w_k)u(k) + D_d(w_k)d(k)+D_{f}(w_k)f(k).
	\end{cases}
\end{align}}
Herein, $u(k)$ represents the input signal, $d(k)$ the exogenous disturbance, $X(k)$ the internal state, $y(k)$ the measured output and $f(k)$ the fault. By defining $z:=[y;u],~x:=[X;d]$ and the parameter-varying polynomial matrices
{\color{black}\begin{align*}
    L(w_k,\shiftop):=\begin{bmatrix}0&B_u(w_k)\\-I&D_u(w_k)\end{bmatrix},\quad F(w_k,\shiftop):=\begin{bmatrix}B_f(w_k)\\D_f(w_k)\end{bmatrix}, \quad 
    H(w_k,\shiftop):=\begin{bmatrix}-G(w_k)\shiftop+A(w_k) & B_d(w_k)\\C(w_k)&D_d(w_k)\end{bmatrix},
\end{align*}}
in~\eqref{eq:DAE}, it can be observed that~\eqref{eq:ldif} is an example of the model description~\eqref{eq:DAE}.

In the absence of a fault signal $f$, i.e., for $f=0$, all possible $z$-trajectories of the system~\eqref{eq:DAE} can be denoted as
\begin{align}\label{eq:behavior1}
    \mathcal{M}(w)\Let\{z: \Z\ra\R^{n_z}|\:\exists x:\Z\ra\R^{n_x}: H(w_k,\shiftop)[x]+L(w_k,\shiftop)[z]=0\},
\end{align}
which is called the \textit{healthy} behavior of the system. For fault detection, the primary objective is to identify whether the trajectory $z$ belongs to this healthy behavior.

 
\section{Design of parameter-varying estimation filter}\label{sec:theorem}
In~\cite{2006Nyberg}, an LTI system, also known as a residual generator, is proposed via the use of an irreducible  polynomial  basis  for  the  nullspace  of $H(w_k,\shiftop)$, denoted  by $N_H(w,\shiftop)$~\footnote{In the remainder of this work, by not explicitly mentioning the time index ``$k$"
 in $w$, we emphasize that the filter coefficients may depend on the parameter signal $w$ in multiple time instances.}. {\color{black}In this work, we take the problem a step further by finding an irreducible polynomial basis $N_H(w,\shiftop)$ for the nullspace of $H(w_k,\shiftop)$, i.e., the state dynamics of an LPV system. Such a polynomial fully characterizes the healthy behavior of the system~\eqref{eq:DAE} as follows:} 
\begin{align}
	\mathcal{M}(w)\!=\!\{z:\Z\ra\R^{n_z}\:|\:N_H(w,\shiftop)L(w_k,\shiftop)[z]\!=\!0\}.\label{eq:beh23}
\end{align}
For the design of an estimation filter, it suffices to introduce a linear combination $N(w_k,\shiftop)=\mu N_H(w_k,\shiftop)$, such that the following objectives for fault detection can be achieved:  
\begin{subequations}\label{eq:lpvtot}
\begin{align}
        a^{-1}(\shiftop)N(w,\shiftop)H(w_k,\shiftop)=&0,\quad\forall w_k\in\mathcal{W},\label{eq:lpvtota}\\
        a^{-1}(\shiftop)N(w,\shiftop)F(w_k,\shiftop)\neq& 0,\quad\forall w_k\in\mathcal{W}.\label{eq:lpvtotb}
\end{align}
\end{subequations}
Here, the polynomial $a(\shiftop)$ is intended to make the estimation filter proper. Moreover, it enables a form of noise attenuation, which is highly recommended towards experimental applications. The above conditions allow us to find a filter to decouple the residual from the time-varying behavior of the system.
In fulfilling the requirements of~\eqref{eq:lpvtot}, a proper LPV estimation filter of the following form can be created:
\begin{align}\label{eq:resgen}
    r\Let a^{-1}(\shiftop)N(w,\shiftop)L(w_k,\shiftop)[z].
\end{align}
 Note, that the degree of $a(\shiftop)$ is not less than the degree of $N(w_k,\shiftop)L(w_k,\shiftop)$ and is stable and that the design of 
 such polynomial is up to the user and can depend on various criteria (e.g., noise sensitivity). 
 In the following lemma, a method to transform the conditions~\eqref{eq:lpvtot} into non-complex, scalar or vector equations is provided, forming a basis for the methodology proposed in the next section.
{\color{black}\begin{Lem}\label{lem:poly}
    Let $N(w,\shiftop)$ be a feasible solution to~\eqref{eq:lpvtot} where the matrices $H(w_k,\shiftop),\:F(w_k,\shiftop),\:a(\shiftop)$ are as in~\eqref{eq:DAE} and~\eqref{eq:lpvtot} and have a particular form of
    \begin{align*}
         H(w_k,\shiftop)=& \sum_{i=0}^{d_H}H_i(w_k)\shiftop^i,& F(w_k,\shiftop)=& \sum_{i=0}^{d_F}F_i(w_k)\shiftop^i,\\
        N(w,\shiftop) =& \sum_{i=0}^{d_N}N_i(w)\shiftop^i,& a(\shiftop)=&\sum_{i=0}^{d_a}a_i\shiftop^i,
    \end{align*}
    where $d_H,\:d_F,\:d_N,\:d_a$ denote the degree of the respective polynomials. Given any parameter signal $w$, the conditions in~\eqref{eq:lpvtot} can be equivalently rewritten as
    \begin{subequations}\label{eq:conditionstot}
    \begin{align}
        \bar{N}(w)\bar{H}(w)=&0,\label{eq:condition1}\\
        \bar{N}(w)\bar{F}(w)\neq&0,\label{eq:condition2}
    \end{align}
    \end{subequations}
    where $\bar{N}(w),\:\bar{H}(w),\:\bar{F}(w)$ are defined as
    \begin{align*}
        \bar{N}(w)\Let&\begin{bmatrix}N_{0}(w)& N_{1}(w)& \hdots&N_{d_{N}}(w)\end{bmatrix},\\
        \bar{H}(w)\Let&\begin{bmatrix}H_{0}(w_{k-d_a})&\hdots&0\\\vdots&\ddots&\vdots\\H_{d_{H}}(w_{k-d_a})&&H_{0}(w_{k-d_a+d_N})\\\vdots&\ddots&\vdots\\0&\hdots&H_{d_{H}}(w_{k-d_a+d_N})\end{bmatrix}^\intercal,\\
        \bar{F}(w)\Let&\begin{bmatrix}F_{0}(w_{k-d_a})&\hdots&0\\\vdots&\ddots&\vdots\\F_{d_{F}}(w_{k-d_a})&&F_{0}(w_{k-d_a+d_N})\\\vdots&\ddots&\vdots\\0&\hdots&F_{d_{F}}(w_{k-d_a+d_N})\end{bmatrix}^\intercal.
    \end{align*}
\end{Lem}}
\begin{proof}
For proving the results of this lemma, observe that~\eqref{eq:resgen} can be rewritten as (using~\eqref{eq:DAE} and~\eqref{eq:lpvtot}):
 \begin{align*}
      r & = -a^{-1}(\shiftop)N(w,\shiftop)F(w_k,\shiftop)[f], \quad \Rightarrow -a(\shiftop)[r] = N(w,\shiftop)F(w_k,\shiftop)[f],\\
     \Rightarrow & -\sum_{h=0}^{d_a}a_h\shiftop^h[r] = \sum_{i=0}^{d_N}N_i(w)\shiftop^i\sum_{j=0}^{d_F}F_j(w)\shiftop^j[f] = \sum_{i=0}^{d_N}\sum_{j=0}^{d_F}N_i(w)F_j(\shiftop^i[w])\shiftop^{i+j}[f].
\end{align*}
Multiplication of both sides with $\shiftop^{-d_a}$, in order to time-shift the relation to result in a  present time residual $r(k)$ as a function of previous faults $f$ and residuals $r$, yields 
\begin{align*}
     -\sum_{h=0}^{d_a}a_h\shiftop^{h-d_a}[r]=&\sum_{i=0}^{d_N}\sum_{j=0}^{d_F}N_i(w)F_j(\shiftop^{i-d_a}[w])\shiftop^{i-d_a+j}[f],
     \end{align*}
for which the right-hand side can be rewritten as
\begin{align*}
    \bar{N}(w)\bar{F}(w)\Big(\begin{bmatrix}\shiftop^{-d_a}I&\shiftop^{1-d_a}I&\hdots&\shiftop^{d_N+d_F-d_a}I\end{bmatrix}[f]\Big).
\end{align*}
This proves the equivalence of~\eqref{eq:condition2} and~\eqref{eq:lpvtotb}. The same line of reasoning applies for proving the equivalence of~\eqref{eq:condition1} and~\eqref{eq:lpvtota}.
\end{proof}
It is worth noting that the matrices $\bar{N}(w)$, $\bar{H}(w)$, and $\bar{F}(w)$ defined in Lemma~\ref{lem:poly} depend on the parameter signal $w$ through $d_a + 1$ consecutive values. That is, at time instant $k$ the filter coefficient $\bar{N}(w)$ depends on $\{w_{k-d_a}, \dots, w_{k}\} \in \mathcal W^{d_a + 1}$. We, however, refrain from explicitly denoting this dependency and simply use the notation of the entire trajectory $w$, say $\bar{N}(w)$. In this light and using the result from Lemma~\ref{lem:poly}, the conditions for fault detectability can be defined as follows.
\begin{Fact}[Conditions of isolability]\label{fact:cond}
    Given the parameter signal $w$, there exists a feasible solution $\bar{N}(w)$ to the conditions~\eqref{eq:conditionstot} if and only if
    \begin{align}\label{eq:iffcond}
        \text{Rank}\left(\left[\bar{H}(w)\quad \bar{F}(w)\right]\right)>\text{Rank}\left(\bar{H}(w)\right).
    \end{align}
\end{Fact}
The proof is omitted as it is a straightforward adaption from~\cite[Fact 4.4]{Esfahani2016}. Using the results from Lemma~\ref{lem:poly}, the main theorem for the LPV fault detection filter can be proposed.
\begin{Prop}[Parameter-varying filter synthesis]\label{theorem:LPVtheorem}
    Let the matrices $\bar{H}(w)$ and $\bar{F}(w)$ be matrices as defined in Lemma~\ref{lem:poly}. Then an LPV fault detection filter of the form~\eqref{eq:resgen} can be found at every time instance $k$, depending on $w_k$, that fulfills~\eqref{eq:lpvtot}, by solving the following convex quadratic program (QP):
    \begin{subequations}
    \label{eq:lpvopt}
    \begin{align}
    \label{eq:lpvopt_obj}
        \bar{N}^*(w)\Let\arg&\min\limits_{\bar{N}}-\lVert\bar{N}\bar{F}(w)\rVert_{\infty}+\lVert\bar{N}\tr\rVert_2^2,\\
    \label{eq:lpvopt_cont}
        &\textnormal{s.t.}\quad \bar{N}\bar{H}(w)=0,
    \end{align}
    \end{subequations}
    where $\lVert\cdot\rVert_\infty$ denotes the supremum norm. 
\end{Prop}
\begin{proof}
 The term in~\eqref{eq:lpvopt_obj}, related to the fault polynomial $\bar{F}(w)$ ensures a maximised sensitivity for the fault, analogous to~\eqref{eq:lpvtota}, whereas the quadratic (regularization) term related to the filter polynomial $\bar{N}(w)$ ensures that the solution to the problem is bounded. The constraint~\eqref{eq:lpvopt_cont} ensures that the effect of unknown disturbances is decoupled from the residual, analogous to the desired filter requirement in~\eqref{eq:lpvtotb}.
\end{proof}
Proposition~\ref{theorem:LPVtheorem} lays the groundwork for creating a estimation filter for an LPV model with measurable scheduling parameters $w$. At first glance, it can appear to be an unattractive solution to solve an optimization problem at each time-step, in order to obtain filter coefficients for the estimation filter. However, we show in the following corollary that a tractable analytical solution can be derived for this problem.
\begin{Cor}[Analytical solution]\label{cor:dualproblem}
    Consider the convex QP optimization problem in~\eqref{eq:lpvopt}. The solution to this optimization problem has an analytical solution given by the following polynomial:
    \begin{align}\label{eq:analyt}
        &\bar{N}_\gamma^*(w) = \frac{1}{2\gamma}\bar{F}_{j^*}\tr(w)(\gamma^{-1}I+\bar{H}(w)\bar{H}\tr(w))^{-1}, \\ &\text{where} \quad j^* = \arg\max\limits_{j \le d_N} \lvert\bar{N}_\gamma^{*}(w)\bar{F}_{j}(w)\rvert,
        \nonumber 
    \end{align}
    where $j$ denotes the $j$-th column of the matrix $\bar{F}(w)$.
    Moreover, the solution~$\bar{N}_\gamma^*(w)$ in~\eqref{eq:analyt} converges to the optimal filter coefficient~\eqref{eq:lpvopt} as the parameter~$\gamma$ tends to $\infty$. For bounded values of $\gamma$,~\eqref{eq:analyt} provides an approximate solution.
\end{Cor}
\begin{proof}
    A dual program of~\eqref{eq:lpvopt} can be obtained by penalizing the equality constraint~\eqref{eq:lpvopt_cont} through a quadratic function as
    \begin{align}\label{eq:dual}
        \sup\limits_{\gamma \ge 0} g(\gamma,w) = \lim_{\gamma \to \infty}g(\gamma,w),
    \end{align}
    where $\gamma\in\mathbb{R}_+$ represents the Lagrange multiplier and $g(\gamma)$ represents the dual function defined as 
    \begin{align*}
        g(\gamma,w) \Let\inf_{\bar{N}}\gamma\lVert\bar{N}\bar{H}(w)\rVert_2^2 + \lVert\bar{N}^{\intercal}\rVert_2^2 - \lVert\bar{N}\bar{F}(w)\rVert_{\infty}.   
    \end{align*}
    Note, that the $\infty$-norm related to the fault sensitivity can temporarily be dropped by viewing the problem~\eqref{eq:lpvopt} and its dual problem~\eqref{eq:dual} as a set of $d_N$ different QPs; note that the matrix $\bar{F}$ has $d_N$ columns. Hence, the set of dual functions is denoted as
    \begin{align}\label{eq:dualfunctions}
        \tilde{g}(\gamma,w) = \inf_{\bar{N}}\underbrace{\gamma\lVert\bar{N}\bar{H}(w)\rVert_2^2 + \lVert\bar{N}^\intercal\rVert_2^2 - \bar{N}\bar{F}(w)}_{\mathcal{L}(\bar{N},\gamma)}.
    \end{align}
    The solution to the convex quadratic dual problem can be found by first finding the partial derivative of the above Langrangian as follows:
    \begin{align*}
        \frac{\partial\mathcal{L}(\bar{N}^*_\gamma(w),\gamma)}{\partial\bar{N}}\!=\!2\gamma\bar{N}^*_\gamma(w)\bar{H}(w)\bar{H}^{\tr}\!(w)\!+\!2\bar{N}^*_\gamma(w)\!-\!\bar{F}^{\tr}\!(w).    
    \end{align*}
    Setting this partial derivative to zero, we arrive at
    \begin{align}
    \bar{N}_\gamma^*(w)=\frac{1}{2\gamma}\bar{F}^\intercal(w)(\gamma^{-1}I+\bar{H}(w)\bar{H}^\intercal(w))^{-1},\label{eq:solutionopt}
    \end{align}
    which provides $d_N$ admissible solutions to the problem with dual functions $\tilde{g}(\gamma,w)$~\eqref{eq:dualfunctions}. The optimal solution is found by choosing the column of $\bar{F}$, such that the fault sensitivity of the filter is maximal, i.e.,
    \begin{align*}
        &\bar{N}_\gamma^*(w) = \frac{1}{2\gamma}\bar{F}_{j^*}\tr(w)(\gamma^{-1}I+\bar{H}(w)\bar{H}\tr(w))^{-1}, \\ &\text{where} \quad j^* = \arg\max\limits_{j \le d_N} \lvert\bar{N}_\gamma^{*}(w)\bar{F}_{j}(w)\rvert,
        \nonumber 
    \end{align*}
    which proves equation~\eqref{eq:analyt}. Substituting this solution back into the dual program~\eqref{eq:dual} yields
    \begin{align*}
        \begin{cases}
           \max\limits_{\gamma}&
           -\frac{1}{4\gamma}\bar{F}_{j^*}^\intercal(w)(\gamma^{-1}I+\bar{H}(w)\bar{H}^\intercal(w))^{-1}\bar{F}_{j^*}(w),\\
            \text{s.t.}\quad &\gamma\geq 0.
        \end{cases}    
    \end{align*}
    This quadratic negative (semi-)definite problem reaches its maximum when $\gamma$ tends to infinity, concluding the proof.
\end{proof}
Considerations for choosing $\gamma$ are elaborated further on inside the algorithmic implementation. For the purpose of estimation, we are particularly interested in a unity zero-frequency gain. In the following corollary, it is shown how to incorporate this condition in the filter. 
\begin{Cor}[Zero steady-state]\label{cor:zerosteadystate}
    {\color{black}Given a filter in the form of in~\eqref{eq:resgen}, where the numerator coefficient $\bar{N}^*_\gamma(w)$ is a solution to the program~\eqref{eq:lpvopt} given analytically in~\eqref{eq:analyt}. The steady-state relation of the mapping $(d,f)\mapsto r$, for any disturbance signal $d$ and a constant fault $f$, is given by
    \begin{align*}
        r=-\frac{\bar{N}^*_\gamma(w)\bar{F}(w)\ind_{{d_N}\times{d_F}}}{\sum_{h=0}^{d_a}a_h}f,    
    \end{align*}
    where $\ind_{{d_N}\times{d_F}}$ denotes a matrix of ones of the dimensions ${d_N}\times{d_F}$. }
    
\end{Cor}
\begin{proof}
    The model equation~\eqref{eq:DAE}, multiplied with a filter $a^{-1}(\shiftop)N(w,\shiftop)$, satisfying the conditions~\eqref{eq:lpvtot}, can be denoted as
    \begin{align*}
        a^{-1}(\shiftop)N(w_k,\shiftop)L(w_k,\shiftop)[z]=&- a^{-1}(\shiftop)N(w_k,\shiftop)F(w_k,\shiftop)[f],\\
        \Rightarrow r=&- a^{-1}(\shiftop)N(w_k,\shiftop)F(w_k,\shiftop)[f],
    \end{align*}
    where the last line is induced by~\eqref{eq:resgen}. The steady-state behavior of this filter can be found by setting $\shiftop=1$, resulting in
    {\color{black}\begin{align*}
        r=&- a^{-1}(1)N(w_k,1)F(w_k,1)[f] = -\frac{\bar{N}(w)\bar{F}(w)\ind_{{d_N}\times{d_F}}}{\sum_{h=0}^{d_a}a_h}f,
    \end{align*}}
    which provides the desired result, hence concluding the proof.
\end{proof}
 {\color{black}Notice, that when the considered system is time-invariant (i.e., $w$ is constant for all $k$), the conclusion from Corollary~\ref{cor:zerosteadystate} coincides with~\cite[Lemma 3.1, Eq. (10)]{van2020multiple}. This condition, together with the stability of $a(\shiftop)$, ensures convergence of the estimation error for piecewise constant fault signals. Let us note that when there is additive unbiased noise on, e.g., the output measurements, the average behavior of the resulting residual will still follow the residual from the deterministic case.} {\color{black} In the next section, we elaborate on the algorithmic implementation of the proposed estimation filter to, e.g., an LPV minimal state-space realization, a non-trivial problem given the potential effects of dynamic dependence~\cite{toth2012}.

\subsection*{Algorithmic Implementation}
For the estimation filter to function according to the objectives~\eqref{eq:lpvtot} (including unity DC gain for estimation), hence preventing any effects from dynamic dependencies~\cite{toth2012}, the filter can be implemented as an LPV Input-Output representation as follows:
\begin{align*}
    r(k) = a_0^{-1}\bar{a}\ind_{d_a}E(w) \begin{bmatrix}z(k-d_a)&\hdots&z(k-d_a+d_N)\end{bmatrix}^{\intercal} = - a_0^{-1}\sum_{i=1}^{d_a}a_{i}r(k-i)
\end{align*}
\begin{align}
    E(w)=&\frac{\bar{N}(w)\bar{L}(w)}{\bar{N}(w)\bar{F}(w)\ind_{d_N\times d_F}},\label{eq:matrixop}
\end{align}
where the matrix $\bar{L}(w)$ is defined as
\begin{align*}
\bar{L}(w)\Let&\begin{bmatrix}L_{0}(w_{k-d_a})&\hdots&0\\\vdots&\ddots&\vdots\\L_{d_{L}}(w_{k-d_a})&&L_{0}(w_{k-d_a+d_N})\\\vdots&\ddots&\vdots\\0&\hdots&L_{d_{L}}(w_{k-d_a+d_N})\end{bmatrix}^\intercal.
\end{align*}
The matrix operation~$E(w)$ in~\eqref{eq:matrixop} ensures the isolation and estimation of the fault, while the filter coefficients in $\bar{a}$ ensure causality of the operation and reduced sensitivity to noise. By substitution of the results from~\eqref{eq:analyt},~\eqref{eq:matrixop} can be rewritten as:
\begin{align*}
    E(w)=\frac{\bar{F}^{\intercal}_{j^*}(\gamma^{-1}I+\bar{H}(w)\bar{H}^{\intercal}(w)))^{-1}\bar{L}}{\bar{F}^{\intercal}_{j^*}(\gamma^{-1}I+\bar{H}(w)\bar{H}^{\intercal}(w)))^{-1}\bar{F}\ind_{d_N\times d_F}},
\end{align*}
from which it can be deduced that the term $\gamma^{-1}I$ solely ensures well-posedness of the involved inversion operations since, based on Proposition~\ref{cor:dualproblem},  ideally $\gamma^{-1}I$ tends to a zero matrix. The filter is well-posed and exact if and only if $\bar{H}(w)$ is of full rank. If this condition is not fulfilled, the analytical solution from~Proposition~\ref{cor:dualproblem} provides a conservative solution (i.e., the filter could inherit a bias), where the Lagrangian operator $\gamma$ needs to be chosen large enough to ensure well-posedness, while being numerically bounded for practical considerations.}


\section{Case study: automated driving}\label{sec:simulation}
 In this section, the proposed method for designing an LPV estimation filter is illustrated based on a fault estimation problem coming from the lateral dynamics of an automated passenger vehicle. A linear bicycle vehicle model is used as a benchmark model~\cite[Equation (1)]{Schmeitz2017} which is controlled in {\em closed-loop} by the same PD control-law as proposed in~\cite{Schmeitz2017}. Within an automotive context it is undesirable to mitigate a fault in closed-loop without being aware of its magnitude. In fact, in the presence of substantial faults, the vehicle is expected to transition to a safe state. This need for estimating the fault shows the applicability of our proposed problem statement in this application context.
First, the model as depicted in Fig.~\ref{fig:simmodel} can formulated as a set of continuous-time linear state-space equations as follows
{\color{black}    \begin{align}
    \footnotesize
    \setlength{\arraycolsep}{2.5pt}
    \medmuskip = 1mu 
        \dot{X}(t)\!\!
        =&\!\!  \underbrace{\begin{bmatrix}\frac{C_f+C_r}{v_x(t)m}&\frac{l_fC_f-l_rC_r}{v_x(t)m}&0&0\\
    \frac{l_fC_f-l_rC_r}{v_x(t)I}&\frac{l_f^2C_f+l_r^2C_r}{v_x(t)I}&0&0\\-1&0&0&v_x(t)\\0&-1&0&0\end{bmatrix}}_{\tilde{A}(v_x)}\!\!\!X(t)\!-\!\!\!\underbrace{\begin{bmatrix}\frac{C_f}{m}\\\frac{l_fC_f}{I}\\0\\0\end{bmatrix}}_{\tilde{B}_u}\!\!u(t)\nonumber
        -\!\!\underbrace{\begin{bmatrix}\frac{C_f}{m}\\\frac{l_fC_f}{I}\\0\\0\end{bmatrix}}_{\tilde{B}_f}\!\!f(t)\!+\!\!\underbrace{\begin{bmatrix}g&0\\0&0\\0&0\\0&v_x(t)\end{bmatrix}}_{\tilde{B}_d(v_x)}\!\!\!\begin{bmatrix}\sin{(\phi(t))}\\\kappa(t)\end{bmatrix},\\
        y(t)=& \begin{bmatrix}0&I\end{bmatrix}X(t),\nonumber
    \end{align}
where the state $X(t)=\begin{bmatrix}\Ddot{\psi}(t)&\dot{v}_y(t)&\dot{y}_e(t)&\dot{\psi}_e(t)\end{bmatrix}$ for which $\dot{\psi}$ denotes the yaw-rate of the vehicle, $v_y$ the lateral velocity of the vehicle, $y_e$ the lateral deviation from the lane center, and, $\psi_e$ the heading deviation from the lane center. The assumed fault, $f$, acts as an additive fault on the input steering angle, $u$.}  
\begin{figure}[t]
\centering
\includegraphics[clip, trim=4.5cm 13.4cm 4.5cm 7.2cm, scale=0.63]{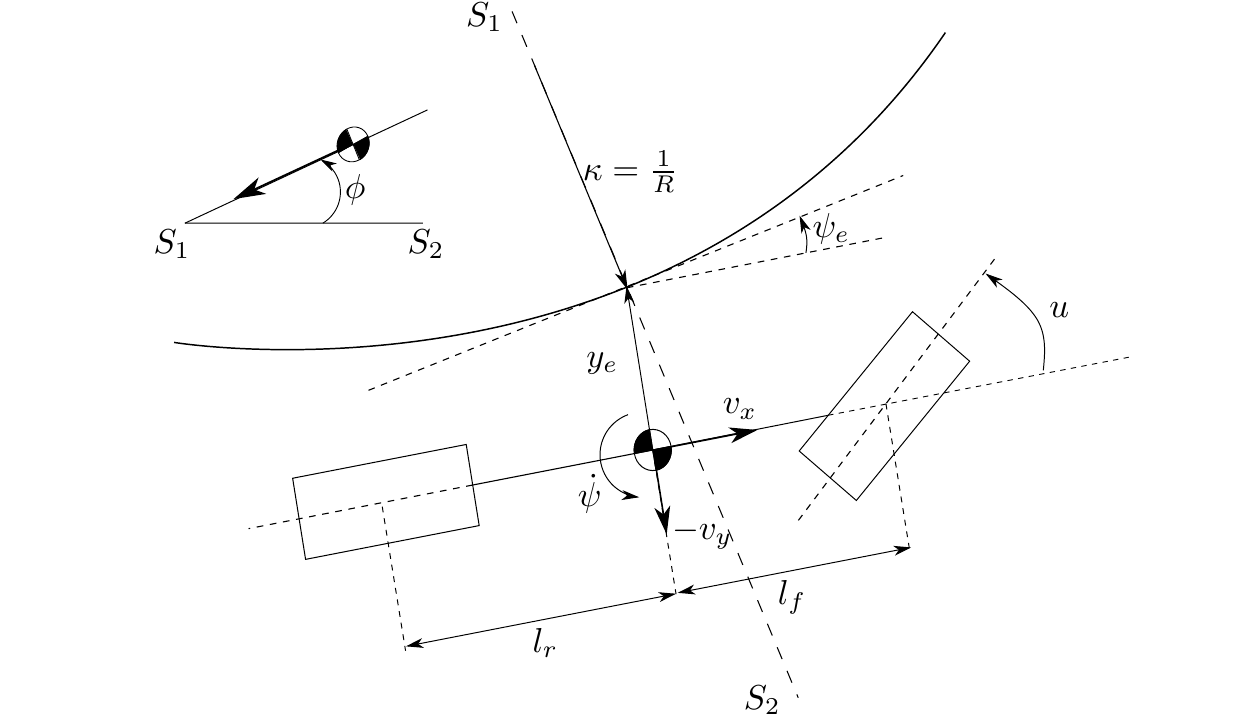}
\caption{Visual representation of the bicycle model.}
\label{fig:simmodel}
\end{figure}
Two disturbances are considered, where $\kappa$ denotes the curvature of the road and $\phi$ denotes the banking angle of the road. The parameters $C_f=1.50\cdot 10^5\:\rm N\cdot rad^{-1}$ and $C_r=1.10\cdot 10^5\:\rm N\cdot rad^{-1}$ represent the lateral cornering stiffness of the front and rear tyres, respectively, $l_f=1.3\:\rm m$ and $l_r=1.7\:\rm m$ represent the distances from the front and rear axle to the center of gravity. It is furthermore assumed that $m=1500\:\rm kg$ represents the total mass of the vehicle, $I=2600\:\rm kg\cdot m^2$ represents the moment of inertia around the vertical axis of the vehicle and $g=9.81\:\rm m\cdot s^{-2}$ represents the gravitational acceleration. The parameter $v_x$ represents the longitudinal velocity of the vehicle and acts as the scheduling parameter $w_k$ from~\eqref{eq:DAE}. {\color{black}The discrete-time system matrices, used for filter synthesis, are found using exact discretization with a sampling time of $h=0.01$s, i.e., $A(v_x)=e^{\tilde{A}(v_x)h}$ and $B(v_x) = \tilde{A}^{-1}(v_x)(A(v_x)-I)\tilde{B}(v_x)$ (for all matrices $\tilde{B}_u, \tilde{B}_f$ and $\tilde{B}_d(v_x)$). The relation from the state to the output remains unchanged through discretization.}

{\color{black}In traffic scenarios it is realistic to assume perturbation in the longitudinal velocity. To capture this, the scheduling parameter is chosen as $v_x(t)=19+5\sin(0.1\pi t)$.} 
{\color{black}The simulation results are generated through Simulink on an Intel Core i7-10850H 2.7 GHz platform. The average computational time needed for evaluating the filter and its output is $8.2\cdot 10^{-5}s$, i.e., a factor 100 lower than the sampling time.} Fig.~\ref{fig:simulation} depicts the simulation results of a 500 sample long scenario. With this simulation, the effectiveness of the LPV estimation filter, using two different sets of filtering coefficients, is shown and compared to an LTI estimation filter (as used in~\cite{van2020multiple}, generated for a velocity $v_x=19\:m/s$). Here, the fault to be estimated is simulated as a realistic abrupt steering wheel offset {\color{black} $f=0.1\frac{\pi}{180}$ radians starting at time sample $k=150$.} Finally, in Fig.~\ref{fig:simulation} the simulation results in two cases, with and without measurement noise are shown. We introduce realistic additive white sensor noises with standard deviations of $\sigma_{\dot{\psi}}=8\cdot 10^{-4}$rad/s, $\sigma_{y_e}=5\cdot10^{-2}$m, $\sigma_{\psi_e}=3\cdot 10^{-3}$ rad. Note, that for the noisy simulations, two different filters are created and depicted. One of which with denominator  $a(\shiftop)=(\shiftop+0.95)^3$. For the other filter (denoted in Fig.~\ref{fig:simulation} as "increased filtering"), $a(\shiftop)=(\shiftop+0.98)^3$ is selected. Note, that the design of $a(\shiftop)$ could highly depend on domain-specific knowledge of the application. For example, $a(\shiftop)$ can be designed to attenuate unmodeled noise or disturbances while respecting the frequency-content of the fault to be estimated.  

\begin{figure}[t]
     \includegraphics[width = 0.8\columnwidth]{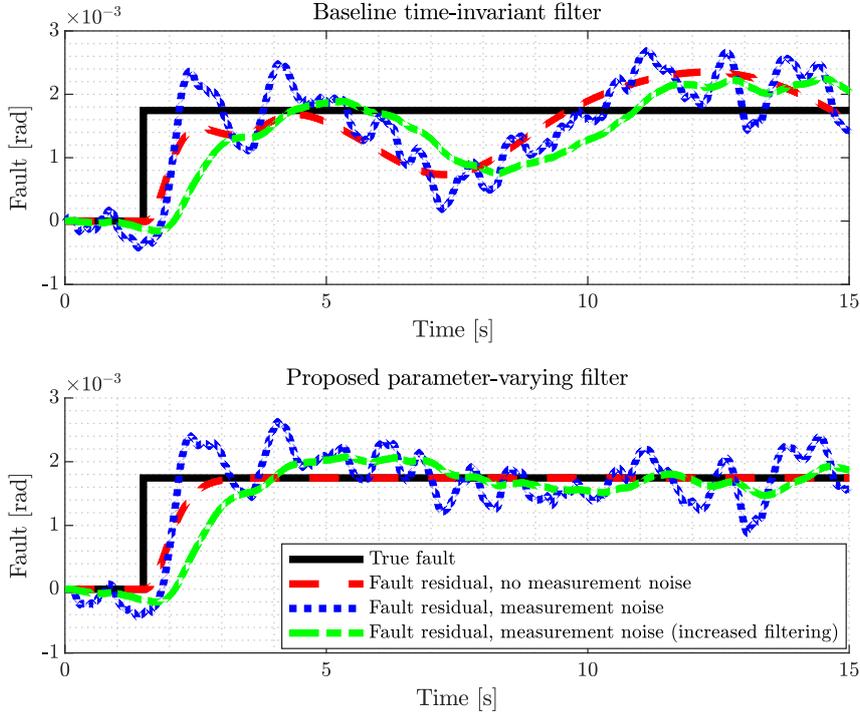}
     \caption{Performance of the time-invariant and the proposed parameter-varying filter in the absence and presence of measurement noise.}
     \label{fig:simulation}
\end{figure}

The results in Fig.~\ref{fig:simulation} show that the baseline LTI filter is not robust against the time-varying longitudinal velocity. Once the fault increases, the residual responds, but does not converge to the true fault. This is explained by the fact that the estimation filter is only designed to decouple unmeasured disturbances from the residual at a constant velocity. Therefore, small effects of disturbances and unmeasured states appear in the residual. In the absence of measurement noise, the LPV filter estimates exactly the injected fault. In the presence of noise, the LPV filter still outperforms the LTI filter, and the estimation accuracy is improved further by placing the poles of numerator $a(\shiftop)$ further towards the exterior of the unit circle. By placing the poles of the numerator $a(\shiftop)$ further towards the origin of the unit-circle, the convergence-rate will increase at the cost of an increased sensitivity for the measurement noise.  


\section{Conclusion and future work}\label{sec:conclusion}
In this paper, a novel synthesis method for a fault estimation filter, applicable a class of discrete-time LPV systems is introduced. The synthesis of such a filter is formulated by an optimization problem as a function of the measurable scheduling parameters, for which a solution exists given a set of (easy-to-check) conditions. We further propose an approximate scheme that can be arbitrarily precise while it enjoys an analytical solution, which supports real-time implementation. The proposed method has been demonstrated on the lateral dynamics of an automated vehicle, showing that in several distinct cases the fault can be estimated. Future work includes the extension to uncertain dynamics, decoupling of non-linearities in the system dynamics and experimental validation of the proposed algorithm.

\bibliographystyle{unsrt}
\bibliography{library/FDI}

\end{document}